\documentclass[12pt,notitlepage]{amsart}
\usepackage{amsmath}     
\usepackage{amscd}

\advance\voffset by -0.4 truein
\advance\hoffset by -0.4 truein
\renewcommand{\O}{\mathcal{O}}
\newcommand{\E}{\mathcal E}
\newcommand{\F}{\mathcal F}
\newcommand{\I}{\mathcal I}
\newcommand{\ox}{\otimes}
\newcommand{\lra}{\longrightarrow}
\renewcommand{\L}{\mathcal L}
\newcommand{\w}{\omega}
\newcommand{\wt}{\widetilde}
\newcommand{\ol}{\overline}

\renewcommand{\:}{\colon}
\renewcommand{\deg}{\mathrm{deg}\,}

\newtheorem{theorem}{Theorem}
\newtheorem{proposition}[theorem]{Proposition}
\theoremstyle{plain}
 
\font\smallrm=cmr8
\font\smallsc=cmcsc10
\font\smallsl=cmsl10

\begin{document}
\author
[{\smallrm EDUARDO ESTEVES}]
{Eduardo Esteves}
\title
[{\smallrm THE STABLE HYPERELLIPTIC LOCUS IN GENUS 3}]
{The stable hyperelliptic locus in genus 3: An application of Porteous Formula}
\begin{abstract} We compute the class of the closure of the locus of 
hyperelliptic curves in the moduli space of stable genus-3 curves in 
terms of the tautological class $\lambda$ and the boundary classes 
$\delta_0$ and $\delta_1$. The expression of this class is known, but 
here we compute it directly, by means of Porteous 
Formula, without resorting to blowups or test curves.
\end{abstract}

\thanks{Supported by 
CNPq, Proc.~303797/2007-0 and 473032/2008-2, and FAPERJ, 
Proc.~E-26/102.769/2008 and E-26/110.556/2010.}

\maketitle

\section{Introduction}

Porteous Formula gives an expression for the class of a degeneracy scheme 
of a map of vector bundles on a smooth variety in terms of the Chern classes 
of the bundles, as long as the 
degeneracy scheme has no excess, that is, has the expected codimension; 
see \cite{F}, Thm.~14.4, p.~254. In \cite{HM}, Question 3.134, p.~169, 
Harris and Morrison posed the question of finding a Porteous-type formula 
for maps between torsion-free sheaves.  

Their question was motivated by the following problem. 
Let $k$ be an algebraically closed field of characteristic different from 2. 
Let $M_3$ be the moduli space of genus-3 smooth, projective, 
connected curves, and 
$\ol M_3$ its compactification by (Deligne--Mumford) stable curves 
over $k$. The vector space $\text{Pic}(M_3)\ox\mathbb Q$ is generated by 
a certain class $\lambda$. This class is the restriction of one 
in $\text{Pic}(\ol M_3)\ox\mathbb Q$, which will be denoted by the 
same symbol. The latter space is generated by $\lambda$, $\delta_0$ and 
$\delta_1$, the latter two being boundary classes.

Let $H\subseteq M_3$ be the locus parameterizing hyperelliptic curves. 
It is a closed subvariety of codimension 1. Let $\ol H$ be its closure 
in $\ol M_3$. We may ask what are the expressions for the class $[H]$ as a 
multiple of $\lambda$ and $[\ol H]$ as a linear combination of $\lambda$, 
$\delta_0$ and $\delta_1$. In \cite{HM}, pp.~162--188, it is shown that

\begin{theorem}\label{thm}
\begin{align*}
[H]=&9\lambda,\\
[\ol H]=&9\lambda-\delta_0-3\delta_1.
\end{align*}
\end{theorem}

(The first formula above had already appeared as a special case of Mumford's 
formula for $[\mathcal H]_Q$ on \cite{Mum}, p.~314.) 

The strategy for obtaining the formula for $[H]$, culminating on page 164 of 
loc.~cit., reviewed in Section 2, was to consider a general family 
of smooth curves, a natural map of vector bundles over its total space 
whose top degeneracy scheme parameterizes the Weierstrass points of 
hyperelliptic fibers of $\pi$, apply Porteous Formula to compute the class of 
this scheme, and then compute the pushforward of this class to the base of the 
family. However, according to \cite{HM}, p.~169, even though ``trying to 
extend the application of Porteous' formula'' to a general family of stable curves 
``is the most obvious approach'' to obtaining a formula for $[\ol H]$, 
the problem is that a certain bundle of jets, namely \eqref{FFF}, 
defined on the smooth locus of the family, ``cannot be extended to a vector bundle over the nodes of 
fibers of the family of curves.''
This motivated 
their question, mentioned above. Not disposing of the asked-for Porteous-type 
formula, they proposed a less direct approach, by means of the so-called 
test curves, culminating with the formula for $[\ol H]$ on page 188 of 
loc.~cit..

In \cite{Diaz} Diaz 
proposes a blowup procedure to obtain a Porteous-type formula 
for maps between torsion-free sheaves. Though no explicit formula is 
produced, as an example, 
the procedure is carried out to obtain a different and more direct 
proof that $[\ol H]=9\lambda-\delta_0-c\delta_1$ for a certain 
$c\in\mathbb Z$, not computable because of excess; see Prop.~1, p.~510 of 
loc.~cit..

In these notes we will see that the bundle of jets 
which ``cannot be extended to a vector bundle'' 
does in fact extend; 
see Section 3. As observed in \cite{HM}, p.~169, the obvious extension, 
namely the sheaf of jets, or principal parts, 
\eqref{notfree}, is not a bundle. However, it becomes so after a pushout 
construction, \eqref{pushout}. 
So we do get a map of vector bundles over the total space $C$ of a 
family of stable curves $C/S$, namely \eqref{nuEF}, 
to which we can consider applying 
Porteous Formula. In contrast to Diaz's approach, no blowups are necessary. 
But, as in \cite{Diaz}, Porteous Formula cannot 
be directly applied because 
of excess. 

Though the excess could be handled in an {\it ad hoc} way, 
we will see that a simple ``twist,'' typical of the theory of 
limit linear series explained in \cite{HM}, Ch.~5, is enough to 
produce a map of vector bundles over $C$, namely \eqref{nu'EF}, 
whose top degeneracy scheme $D$ has the expected 
codimension by Proposition \ref{finite}, the class of which can thus be 
computed by Porteous Formula. Its pushforward to $S$ is given by 
Proposition~\ref{porteous}.

As in \cite{Diaz}, $D$ comprises more than the Weierstrass 
points of the (smooth) hyperelliptic fibers of $\pi$. To get the formula 
for $[\ol H]$ of Theorem \ref{thm}, the excess points 
must be counted out. To remove them, we need to establish 
their multiplicities in $D$. This is Proposition \ref{mult}. The three 
Propositions imply the Theorem.

Though sheaves of jets, or principal parts, for a family of stable curves 
are not vector bundles, 
there are vector bundle substitutes that agree with them on the smooth locus 
of the family. 
These substitutes have appeared, in various degrees of generality, in 
\cite{E1}, \cite{E2}, \cite{G1}, \cite{G2}, \cite{LT1} and \cite{LT2}. It may 
thus well be that a Porteous-type formula for maps between torsion-free 
sheaves is not necessary for dealing with stable curves. What 
would certainly be useful instead, 
is a way of dealing with excessive degeneracy schemes, 
a phenomenon typical in enumerative questions.

The layout of these notes is the following: In Section 2 we review the 
approach in \cite{HM} for computing $[H]$. In Section 3, we show how 
to produce a map of vector bundles over the total space $C$ of a general 
family of stable curves $\pi\:C\to S$, whose top degeneracy scheme $D$ has the 
expected codimension, and contains the eight Weierstrass points on each 
(smooth) hyperelliptic fiber, among others on singular fibers; see 
Proposition \ref{finite}. In Section 4, we apply Porteous Formula to compute 
the pushforward of the class of $D$ to $S$; see Proposition \ref{porteous}. 
To conclude the proof of Theorem~\ref{thm}, in Section 5, 
we compute the multiplicity with which the points 
on singular fibers appear in $D$, our Proposition \ref{mult}, thereby 
finishing the computation of $[\ol H]$.

These notes report on work started during a visiting professorhip 
of the author to the Universit\`a degli Studi di Torino. The 
author would like to thank Regione Piemonte for 
financing his position. Also, he would 
like to thank the Dipartimento di Matematica of the University, specially 
Prof.~Alberto Conte and Prof.~Marina Marchisio, for the warm hospitality 
extended. Finally, he would like to thank Letterio Gatto for many inspiring 
discussions on the subject.

\section{Smooth curves}

The formulas in Theorem \ref{thm} 
can be obtained by considering the Picard group of the 
functor, or that of the associated stack. Given a 
flat, projective map $\pi\:C\to S$ with 
smooth, connected fibers of dimension 1 and genus 3, 
the class $[H]$ corresponds to a 
class $h^\pi$ on $S$, and $\lambda$ to a class $\lambda^\pi$, the 
first Chern class of the rank-3 locally free sheaf 
$\pi_*\Omega^1_{C/S}$. To show that $[H]=9\lambda$, it is 
equivalent to show that $h^\pi=9\lambda^\pi$ for every such $\pi$. 

Likewise, to show that $[\ol H]=9\lambda-\delta_0-\delta_1$, we 
may consider the corresponding classes $\ol h^\pi$, $\lambda^\pi$, 
$\delta_0^\pi$ and $\delta_1^\pi$ on the target of 
flat, projective maps $\pi\:C\to S$ 
whose fibers are stable curves of (arithmetic) genus 3, and show that 
\begin{equation}\label{olh}
\ol h^\pi=9\lambda^\pi-\delta_0^\pi-3\delta_1^\pi.
\end{equation}
Since $\ol M_3$ is stackwise smooth and complete, we need only consider 
one-dimensional smooth projective targets $S$, 
and show the above formula holds 
for a general such map $\pi$. Alternatively, 
we may show that the degrees on both sides of the fomula are equal 
for three special such maps $\pi$, yielding three linearly independent 
triples $(\deg(\lambda^\pi),\deg(\delta_0^\pi),\deg(\delta^1_\pi))$. The 
latter is known as the method of test curves, employed in \cite{HM}.

The strategy to compute $[H]$ in \cite{HM} is as follows. 
Let $\pi\:C\to S$ be a projective, flat map between schemes 
of finite type over the algebraically closed field $k$. Assume 
the (geometric) fibers of $\pi$ are smooth, connected curves of genus 3. 
Given a (closed) point $s\in S$, we will let 
$C_s:=\pi^{-1}(s)$. 

Let $\Omega^1_{C/S}$ be the relative 
cotangent sheaf. Set $\E:=\pi^*\pi_*\Omega^1_{C/S}$. 
At a (closed) point $P$ of $C$, letting $s:=\pi(P)$, we have
$$
\E|_P=H^0(C_s,\Omega^1_{C_s/k}).
$$
Since the fibers of $\pi$ have genus three, 
$\E$ is locally free of rank 3.

Let $\F$ be the relative bundle of first-order jets, or 
principal parts, of $\Omega^1_{C/S}$. 
In other words,
\begin{equation}\label{FFF}
\F:=p_{1*}\Big(p_2^*\Omega^1_{C/S}\ox\frac{\O_{C\times_SC}}
{\I_{\Delta|C\times_S C}^2}\Big),
\end{equation}
where $p_i\:C\times_SC\to C$ is the projection onto the indicated factor, 
for $i=1,2$, and $\Delta\subset C\times_SC$ is the diagonal. 
At a point $P\in C$, letting $s:=\pi(P)$, we have 
$$
\F|_P:=H^0(C_s,{\Omega^1_{C_s|k}}_{\Big|{\text{Spec}\Big(\frac{\O_{C_s,P}}
{\mathfrak m_{C_s,P}^2}\Big)}}),
$$
were $\mathfrak m_{C_s,P}$ is the maximal ideal of the local ring 
$\O_{C_s,P}$. Since $C_s$ is smooth, $\F$ is locally free of rank 2. 

Since $\pi^*\pi_*\Omega^1_{C/S}=p_{1*}p_2^*\Omega^1_{C/S}$, there is a 
natural map $\nu\:\E\to\F$. At a given $P\in C$, the map is an 
evaluation map: Given a local parameter $t$ of $C_s$ at $P$, where 
$s:=\pi(P)$, using $dt$ to trivialize $\Omega^1_{C_s|k}$ at $P$, 
the map $\nu|_P$ assigns to a global differential form, 
whose germ at $P$ can be written as $fdt$ for $f\in\O_{C_s,P}$, 
the class $(f(0)+f'(0)t)dt\text{ (mod }t^2dt\text{)}$, 
where $f':=df/dt$. Thus, $\nu|_P$ fails to be surjective if and only if 
$H^0(C_s,\Omega^1_{C_s|k}(-2P))$ fails to have codimension 2 
in $H^0(C_s,\Omega^1_{C_s|k})$, that is, if and only if $C_s$ 
is hyperelliptic, and $P$ is a Weierstrass point of $C_s$. 

Let $D$ be the top degeneracy scheme of the map $\nu$, supported on the 
set of points $P\in C$ where $\nu|_P$ fails to be surjective. 
Then, assuming the general fiber of $\pi$ is nonhyperelliptic, 
$D$ has codimension at least 2. This is however the expected codimension, 
thus the actual codimension. Then, assuming $S$ is smooth, or at least 
Cohen--Macaulay, Porteous Formula (\cite{F}, Thm.~14.4, p.~254 or 
\cite{HM}, Thm.~3.114, p.~161) gives an expression for $[D]$:
$$
[D]=c_2(\E^*-\F^*)\cap [C].
$$

There are 8 Weierstrass points on a hyperelliptic curve of genus 3. Thus
$$
8h^\pi=\pi_*[D]=\pi_*(c_2(\E^*-\F^*)\cap [C]).
$$
If we compute the right-hand side of the formula above we will get 
$72\lambda^\pi$. We will not do this here, as in Sections 3, 4 and 5 
we will use 
the same procedure to prove the more general \eqref{olh}.

\section{Stable curves}

In \cite{HM}, the expression for $[\ol H]$ is computed by the method of 
test curves. Instead, we will use a ``twist'' of the same method used to 
compute $[H]$. 

Let $\pi\:C\to S$ be a projective, flat map between schemes of finite 
type over the algebraically closed field $k$. Assume 
the fibers of $\pi$ are stable curves of genus 3.

As pointed out in \cite{HM}, the relative cotangent sheaf 
$\Omega^1_{C/S}$ is not locally free, but can be replaced by the 
(invertible) relative dualizing sheaf $\omega_{C/S}$, 
equal to the cotangent sheaf 
away from the nodes of the fibers. The restriction 
$\w_{C_s}$ of $\w_{C/S}$ to a fiber $C_s$ is Rosenlicht's sheaf of 
regular differential forms, those being the meromorphic forms regular 
everywhere, except over a node, where the form must have at most 
simple pole at each branch with zero residue sum.  

Set 
$\E:=\pi^*\pi_*\omega_{C/S}$. As in Section 2, the sheaf $\E$ is 
locally free of rank 3. On the other hand, \cite{HM} asserts that the 
restriction to the smooth locus of $C/S$ of 
\begin{equation}\label{notfree}
p_{1*}\Big(p_2^*\omega_{C/S}\ox\frac{\O_{C\times_SC}}
{\I_{\Delta|C\times_S C}^2}\Big),
\end{equation}
where, as before, $p_i\:C\times_SC\to C$ is the 
projection onto the indicated factor, 
for $i=1,2$, and $\Delta\subset C\times_SC$ is the diagonal, 
does not extend to a locally free sheaf 
on the whole $C$. This is false. 

It is true, as pointed out in \cite{HM}, that \eqref{notfree}, the 
sheaf of first-order principal parts of $\omega_{C/S}$ is not locally 
free, but there is a locally free substitute that coincides with that 
sheaf away from the nodes. 
In the case at hand, it is easy to produce the substitute, 
by considering the pushout for the following diagram of maps:
$$
\begin{CD}
0 @>>> \omega_{C/S}\ox\Omega^1_{C/S} @>>> 
p_{1*}\Big(p_2^*\omega_{C/S}\ox\frac{\O_{C\times_SC}}
{\I_{\Delta|C\times_S C}^2}\Big) @>>> \omega_{C/S} @>>> 0\\
@. @VVV @. @. @.\\
@. \omega_{C/S}^{\ox 2}
\end{CD}
$$
where the exact sequence is obtained from the natural exact sequence
$$
0 \lra \frac{\I_{\Delta|C\times_S C}}{\I_{\Delta|C\times_S C}^2} \lra
\frac{\O_{C\times_SC}}{\I_{\Delta|C\times_S C}^2} \lra 
\frac{\O_{C\times_SC}}{\I_{\Delta|C\times_S C}} \lra 0,
$$
and the vertical map is obtained from the ``canonical class'' 
$\Omega^1_{C/S}\to\omega_{C/S}$ by tensoring with $\omega_{C/S}$. 
The pushout construction completes the above diagram to a map of 
short exact sequences:
\begin{equation}\label{pushout}
\begin{CD}
0 @>>> \omega_{C/S}\ox\Omega^1_{C/S} @>>> 
p_{1*}\Big(p_2^*\omega_{C/S}\ox\frac{\O_{C\times_SC}}
{\I_{\Delta|C\times_S C}^2}\Big) @>>> \omega_{C/S} @>>> 0\\
@. @VVV @VVV @| @.\\
0 @>>> \omega_{C/S}^{\ox 2} @>>> \F @>>> \omega_{C/S} @>>> 0
\end{CD}
\end{equation}
Of course, $\F$ is locally free of rank 2.

As before, there is a natural homomorphism from $\E$ to \eqref{notfree}, 
which can be composed to a homomorphism 
\begin{equation}\label{nuEF}
\nu\:\E\to\F.
\end{equation}
It restricts, 
over the open subset of points of $S$ parameterizing smooth fibers, to 
the map of locally free sheaves considered in Section 2. In fact, away 
from the nodes, the description is the same: If $P\in C$ is not a node 
of $C_s$, where $s:=\pi(P)$, then, considering a local parameter $t$ of 
$C_s$ at $P$, and using $dt$ to trivialize 
$\Omega^1_{C_s|k}$ at $P$, 
the map $\nu|_P$ assigns to a regular differential form, 
whose germ at $P$ can be written as $fdt$ for $f\in\O_{C_s,P}$, 
the class $(f(0)+f'(0)t)dt\text{ (mod }t^2dt\text{)}$, 
where $f':=df/dt$. 

Assume now that $S$ is one-dimensional and $\pi$ is ``general.'' 
More precisely, assume $C$ is nonsingular, the general fiber of $\pi$ 
is smooth, and the (finitely many) singular fibers have only one 
singularity. There are thus two types of singular fibers: 
an irreducible curve $Z$ with a node 
whose two branches on the normalization $\wt Z$ are in general 
position with respect to the canonical system; and the union of 
a genus-1 curve $X$ with a genus-2 curve $Y$ meeting transversally at 
a point $N$ which is not a Weierstrass point of $Y$. From the arguments 
we will use, we may harmlessly assume that there are just two singular 
fibers, $Z$ and $X\cup Y$. 

The top degeneracy scheme of $\nu$ has dimension 1. Indeed, it contains 
the curve $X$. This is because $\w_{C/S}|_X=\O_X(N)$, whose space of 
global sections has dimension 1. Since the dimension of the degeneracy scheme 
is not the 
minimum possible, the so-called expected dimension, 
which in this case is 
zero, we say that we have excess. To deal with excess in Intersection 
Theory is generally hard. In our case, the excess could be handled in an 
\emph{ad hoc} way. We will nonetheless avoid it.

What we will do is replace $\w_{C/S}$ by a ``twisted'' sheaf: 
$$
\L:=\w_{C/S}(-X):=\w_{C/S}\ox\O_C(-X).
$$
We will set $\E':=\pi^*\pi_*\L$, and let $\F'$ be the sheaf obtained 
by the pushout construction, part of the map of short exact 
sequences:
\begin{equation}\label{pout}
\begin{CD}
0 @>>> \L\ox\Omega^1_{C/S} @>>> 
p_{1*}\Big(p_2^*\L\ox\frac{\O_{C\times_SC}}
{\I_{\Delta|C\times_S C}^2}\Big) @>>> \L @>>> 0\\
@. @VVV @VVV @| @.\\
0 @>>> \L\ox\omega_{C/S} @>>> \F' @>>> \L @>>> 0
\end{CD}
\end{equation}
As before, there is a natural map 
\begin{equation}\label{nu'EF}
\nu'\:\E'\to\F'.
\end{equation}

\begin{proposition}\label{finite} 
The sheaves $\E'$ and $\F'$ are locally free 
of ranks $3$ and $2$, respectively. In addition, the top degeneracy scheme 
of $\nu'$ is zero-dimensional, and consists of the following 
points:
\begin{enumerate}
\item The $8$ Weierstrass points of each smooth hyperelliptic fiber 
of $\pi$;
\item The node of $Z$;
\item The node $N$ of $X\cup Y$;
\item The $3$ points $A\in X-\{N\}$ such that $2A\equiv 2N$; 
\item The $6$ Weierstrass points of $Y$.
\end{enumerate}
\end{proposition}

\begin{proof} That $\F'$ is locally free of rank 2 follows immediately 
from it 
being the middle sheaf in the bottom exact sequence of Diagram~\eqref{pout}.
On the other hand, 
to show that $\E'$ is locally free of rank 3, we need to show that 3 is 
the dimension of the space of global sections of $\L$ restricted to each 
fiber. This is clear except for $\L|_{X\cup Y}$. Now, 
there are two exact sequences associated to 
$\L|_{X\cup Y}$:
\begin{equation}\label{two}
\begin{aligned}
0\to\L|_X(-N)\to&\L|_{X\cup Y}\to\L|_Y\to 0,\\
0\to\L|_Y(-N)\to&\L|_{X\cup Y}\to\L|_X\to 0.
\end{aligned}
\end{equation}
The first yields the exact sequence
\begin{equation}\label{first}
0\to\O_X(N)\to\L|_{X\cup Y}\to\w_Y\to0,
\end{equation}
where $\w_Y$ is the canonical sheaf of $Y$. It follows from 
the long exact sequence in cohomology associated to \eqref{first} 
that 
$h^0(X\cup Y,\L|_{X\cup Y})=3$. So, $\E'$ is locally free of rank 3. 

Let $D$ denote the top degeneracy scheme of $\nu'$. Since 
$h^0(X\cup Y,\L|_{X\cup Y})=3$, it follows from the base-change theorem 
and the exactness of \eqref{first} that the two maps in the 
composition below are surjective:
$$
H^0(C,\L) \lra H^0(X\cup Y,\L|_{X\cup Y}) \lra H^0(Y,\L|_Y).
$$
Since $\L|_Y\cong\w_Y$, and since $H^0(Y,\w_Y(-2A))\neq 0$ if and only if 
$A$ is a Weierstrass point of $Y$, it follows that 
$D\cap Y-\{N\}$ consists of the six Weierstrass points of $Y$.

The second exact sequence in \eqref{two} yields the exact 
sequence:
$$
0\lra\w_Y(-N)\lra\L|_{X\cup Y}\lra\O_X(2N)\lra 0.
$$
Since $h^0(Y,\w_Y(-N))=1$ and $h^0(X,\O_X(2N))=2$, it follows as 
before that the two maps in the 
composition below are surjective:
$$
H^0(C,\L) \lra H^0(X\cup Y,\L|_{X\cup Y}) \lra H^0(X,\L|_X).
$$
Thus, since $\L|_X\cong\O_X(2N)$, it follows that $D\cap X-\{N\}$ 
consists of the three points $A\in X-\{N\}$ such that $2A\equiv 2N$.

As explained in Section 2, 
$D$ does not intersect nonhyperelliptic fibers, and 
intersects each hyperelliptic fiber in its 8 Weierstrass points. 

As for the fiber $Z$, consider the following 
composition of natural maps:
$$
H^0(C,\L)\lra H^0(Z,\L|_Z)\lra H^0(\wt Z,\L|_{\wt Z}).
$$
The first map above is surjective by the base-change theorem. The second 
is the pullback map to the normalization, which is injective, and so an 
isomorphism, since both source and target have dimension 3. Here we used the 
Riemann--Roch Theorem and the fact that $\L|_{\wt Z}=\w_{\wt Z}(M_1+M_2)$, 
where $\w_{\wt Z}$ is the canonical sheaf of $\wt Z$, and $M_1$ and $M_2$ are the 
two points over the node $M$ of $Z$. Assuming $M_1$ and $M_2$ are in general 
position, $M_1+M_2$ does not move, and thus, by the Riemann--Roch Theorem, 
$h^0(\wt Z,\w_{\wt Z}(M_1+M_2-2B))=1$ 
for every $B\in\wt Z$ distinct from $M_1$ and 
$M_2$. Hence $D$ intersects $Z$ at most at its node. 

It 
remains only to show that the nodes of $Z$ and of $X\cup Y$ belong to 
$D$, but this will be done in Proposition \ref{mult}. 
What is important for what 
follows is that we have already shown $D$ to be finite.
\end{proof}

\section{An application of Porteous Formula}

\begin{proposition}\label{porteous} 
Let $D$ be the top degeneracy scheme of 
$\nu'\:\E'\to\F'$. Then
$$
\pi_*[D]=72\lambda^\pi-7\delta_0^\pi-7\delta_1^\pi.
$$
\end{proposition}

\begin{proof} Since, by 
Proposition \ref{finite}, $D$ has the right codimension, we may compute 
its class in $C$ by Porteous Formula 
(\cite{F}, Thm.~14.4, p.~254 or \cite{HM}, Thm.~3.114, p.~161):
$$
[D]=c_2({\E'}^*-{\F'}^*)\cap [C].
$$
Expanding, 
\begin{equation}\label{[D]}
\begin{aligned}
{[D]}=&\Big[\frac{c({\E'}^*)}{c({\F'}^*)}\Big]_2\cap[C]\\
=&\Big[\frac{1-c_1(\E')+c_2(\E')}{1-c_1(\F')+c_2(\F')}\Big]_2\cap[C]\\
=&\big[(1-c_1(\E')+c_2(\E'))(1+c_1(\F')-c_2(\F')+c_1(\F')^2)\big]_2
\cap [C]\\
=&(c_2(\E')-c_1(\E')c_1(\F')+c_1(\F')^2-c_2(\F'))\cap [C].
\end{aligned}
\end{equation}

Now, $\F'$ sits in the middle of an exact sequence of the form:
$$
0\lra\L\ox\omega_{C/S}\lra\F'\lra\L\lra 0.
$$
By the Whitney Sum Formula (\cite{F}, Thm.~3.2(e), p.~50), 
letting $K:=c_1(\omega_{C/S})\cap [C]$, 
\begin{equation}\label{cF}
\begin{aligned}
c_1(\F')\cap [C]=&(2c_1(\L)+c_1(\omega_{C/S}))\cap [C]=3K-2[X],\\
c_2(\F')\cap [C]=&c_1(\L)c_1(\L\ox\w_{C/S})\cap [C]=(K-[X])(2K-[X]).
\end{aligned}
\end{equation}

On the other hand, from the exact sequence
$$
\begin{CD}
0 @>>> \w_{C/S}(-X) @>\cdot X>> \w_{C/S} @>>> \w_{C/S}|_X @>>> 0,
\end{CD}
$$
since $\w_{C/S}|_X=\O_X(N)$ and $H^1(X,\O_X(N))=0$, 
we get the long exact sequence
$$
\begin{CD}
0 @>>> \pi_*(\w_{C/S}(-X)) @>>> \pi_*\w_{C/S} @>\beta >> \pi_*\O_X(N)\\
@>>> R^1\pi_*(\w_{C/S}(-X)) @>\gamma >> R^1\pi_*(\w_{C/S}) @>>> 0.
\end{CD}
$$
As we have seen 
in the proof of Proposition \ref{finite}, 
$\pi_*(\w_{C/S}(-X))$ is a locally free sheaf of rank 3 with 
formation commuting with base change, whence 
$R^1\pi_*(\w_{C/S}(-X))$ is invertible. Since so is 
$R^1\pi_*(\w_{C/S})$, it follows that $\gamma$ is an isomorphism. 
So $\beta$ is surjective. Since $h^0(X,\O_X(N))=1$, it follows 
that
$$
c_1(\pi_*(\w_{C/S}(-X)))=c_1(\pi_*\w_{C/S})-\delta_1=\lambda^\pi-\delta_1^\pi.
$$
Thus 
\begin{equation}\label{cE}
\begin{aligned}
c_1(\E')\cap [C]=&\pi^*(\lambda^\pi-\delta_1^\pi),\\
c_2(\E')\cap [C]=&0.
\end{aligned}
\end{equation}

Replacing \eqref{cF} and \eqref{cE} in \eqref{[D]}, we get
$$
[D]=\pi^*(\delta_1^\pi-\lambda^\pi)(3K-2[X])+(3K-2[X])^2-(K-[X])(2K-[X])
$$
Thus,
\begin{align*}
\pi_*[D]=&(\delta_1^\pi-\lambda^\pi)\pi_*(3K-2[X])+
\pi_*((3K-2[X])^2)-\pi_*((K-[X])(2K-[X]))\\
=&12(\delta_1^\pi-\lambda^\pi)+\pi_*(9K^2-16[N]+4[X]\pi^*\delta_1^\pi)-
\pi_*(2K^2-4[N]+[X]\pi^*\delta_1^\pi)\\
=&12(\delta_1^\pi-\lambda^\pi)+9\kappa^\pi-16\delta_1^\pi
-2\kappa^\pi+4\delta_1^\pi\\
=&7\kappa^\pi-12\lambda^\pi,
\end{align*}
where $\kappa^\pi:=\pi_*K^2$. In the first equality above we used the 
projection formula. In the second, we used that $\pi$ collapses $X$ to 
a point, that $\w_{C/S}|_X\cong\O_X(N)$, that 
$[X]=\pi^*\delta_1^\pi-[Y]$, that $\pi_*K=4[S]$, and that 
$[X][Y]=[N]$. In the third, we used again the 
projection formula and the fact that $\pi$ collapses $X$ to a point. 

Finally, using that (\cite{HM}, (3.110), p.~158)
$$
\kappa^\pi=12\lambda^\pi-\delta_0^\pi-\delta_1^\pi,
$$
we get
$$
\pi_*[D]=7(12\lambda^\pi-\delta_0^\pi-\delta_1^\pi)-12\lambda^\pi,
$$
which yields the stated formula.
\end{proof} 

\section{Multiplicities}

Once we remove the contribution of the points of items (2) to (5) 
of Proposition~\ref{finite} from the expression of $\pi_*[D]$ in 
Proposition \ref{porteous} we get $8\ol h^\pi$, and thus 
\eqref{olh}. This means that the node of $Z$ should appear in $D$ with 
multiplicity 1, and the remaining points, items (3) to (5), 
should count to 17, with multiplicitites. Indeed:

\begin{proposition}\label{mult} The scheme $D$ consists of:
\begin{enumerate}
\item the $8$ Weierstrass points of each smooth hyperelliptic fiber 
of $\pi$, each with multiplicity $1$;
\item\label{mult2} the node of $Z$, with multiplicity $1$;
\item\label{mult3} the node $N$ of $X\cup Y$, with multiplicity $2$;
\item\label{mult4} 
the $3$ points $A\in X-\{N\}$ such that $2A\equiv 2N$, each 
with multiplicity $1$; 
\item\label{mult5} 
the $6$ Weierstrass points of $Y$, each with multiplicity $2$.
\end{enumerate}
\end{proposition}

\begin{proof} We have seen in Proposition~\ref{finite} that $D$ 
consists at most of the points listed above. 
The multiplicity of a Weierstrass point of a smooth 
hyperelliptic 
fiber of $\pi$ has been established in \cite{HM}, Ex.~3.116, p.~164. 
The multiplicity of 
the node of $Z$ can be computed in essentially the same way as the 
multiplicity computation in \cite{Diaz}; we will do it at the end 
for the sake of completeness.

We will first establish the multiplicities at smooth points, starting 
with \eqref{mult4}. 
Let $A\in X-N$ such that $2A\equiv 2N$. Let $t$ be a local parameter of 
$\O_{S,s}$, where $s:=\pi(A)$. Since $\pi$ is smooth at $A$, 
there is $u\in\O_{C,A}$ 
such that $t,u$ form a regular system of parameters 
for $\O_{C,A}$. Set $\L':=\L(-X)$. The exactness of the natural sequence 
$$
0\to\L'\to\L\to\L|_X\to 0,
$$
coupled with the surjectivity of the restriction map 
$H^0(C,\L)\to H^0(X,\L|_X)$ shown in the proof of 
Proposition~\ref{finite}, yields the exactness of
$$
0\to H^0(C,\L')\to H^0(C,\L)\to H^0(X,\L|_X)\to0.
$$
Since $\L|_X\cong\O_X(2N)$, we may choose a $k[[t]]$-basis $s_1,s_2,s_3$ of $H^0(C,\L)$ 
such that $\psi:=s_1$ generates $\L_A$ and
      \begin{align*}
	s_2\in& (u^2+(t,u^3))\psi,\\
	s_3\in& (t)\psi.
      \end{align*}

Now, since $\L'=\L(-X)$, it follows that $ts_1,ts_2,s_3\in H^0(C,\L')$. 
Moreover, as 
$h^0(X,\L|_X)=2$, they form a 
$k[[t]]$-basis of $H^0(C,\L')$. But $\L'|_X\cong\O_X(3N)$ and, since 
$H^0(Y,\w_Y(-2N))=0$, the restriction map 
$H^0(C,\L')\to H^0(X,\L'|_X)$ is an isomorphism. 
Since $3A\not\equiv 3N$, the vanishing orders of 
sections of $\O_X(3N)$ at $A$ are $0,1,2$. Since $\L'_A=\O_{C,A}t\psi$, the vanishing orders of 
$ts_1|_X$ and $ts_2|_X$ at $A$ 
are 0 and 2, respectively. Thus, replacing $s_3$ by $s_3-(s_3/ts_1)(A)ts_1$, 
we may assume that the vanishing order of $s_3|_X$ at $A$ is 1. So, we 
may assume
      \begin{align*}
	s_1\in& \{\psi\},\\
	s_2\in& (u^2+(t,u^3))\psi,\\
	s_3\in& (u+(t,u^2))t\psi.
	\end{align*}

Thus, $D$ is the zero scheme 
given by the maximal minors of a matrix whose 
entries are in the corresponding entries of
    $$\left[\begin{matrix} \{1\} & u^2+(t,u^3) & tu+(t^2,tu^2)\\ \{0\} & 
2u+(t,u^2) & t+(t^2,tu)) \end{matrix}\right]$$
Since the characteristic of the ground field $k$ is assumed different 
from 2, it follows that the multiplicity of $D$ at $A$ is 1.

We will establish the multiplicity in \eqref{mult5}
now. Let $A$ be a Weierstrass point of $Y$. 
Let $t$ be a local parameter of 
$\O_{S,s}$, where $s:=\pi(A)$. Since $\pi$ is smooth at $A$, 
there is $u\in\O_{C,A}$ such that 
$t,u$ form a regular system of parameters 
for $\O_{C,A}$. 

Let $\L':=\L(-Y)$ and $\L'':=\L(-2Y)$. Then $\L'_A=t\L_A$ and 
$\L''_A=t^2\L_A$. As we have seen in the proof of 
Proposition \ref{finite}, the restriction map
\begin{equation}\label{LLL}
H^0(C,\L)\lra H^0(Y,\L|_Y)
\end{equation}
is surjective. Since $\L|_Y\cong\w_Y$, and $A$ is a Weierstrass point of 
$Y$, there is a basis $s_1,s_2,s_3$ of $H^0(C,\L)$ such that 
$s_1(A)\neq 0$, $s_2|_Y$ vanishes at $A$ with multiplicity 2, and 
$s_3|_Y=0$. Let $\psi$ be the germ of $s_1$ at $A$. So, in $\L_A$, we may 
assume
\begin{align*}
s_1\in&\{\psi\},\\
s_2\in&(u^2+(t,u^3))\psi,\\
s_3\in&(t)\psi.
\end{align*}

Since \eqref{LLL} is surjective, and since $h^0(Y,\L|_Y)=2$, the sections 
$ts_1,ts_2,s_3$ of $H^0(C,\L')$ form a basis. Since 
$\L'|_X\cong\O_X(N)$ and $\L'|_Y\cong\w_Y(N)$, and since 
$h^0(X,\O_X(N))=1$ and $h^0(Y,\w_Y(N))=2$, it follows that the 
restriction map
$$
H^0(C,\L')\lra H^0(Y,\L'|_Y)
$$
is surjective, and $h^0(Y,\L'|_Y)=2$. Thus there is $s'\in H^0(C,\L')$ forming a 
$k[[t]]$-basis of $H^0(C,\L')$ together with $ts_1,ts_2$ such that 
$s'|_Y=0$. Up to replacing $s_3$ by $s_3-fts_1-gts_2$ for $f,g\in k[[t]]$, 
we may assume $s'=s_3$. So $s_3\in (t)^2\psi$.

In addition, it follows that 
the sections $t^2s_1,t^2s_2,s_3$ of $H^0(C,\L'')$ form a basis. 
But $\L''|_Y\cong\w_Y(2N)$. Since $A$ is a Weierstrass point 
of $Y$, and $Y$ has genus~2,
\begin{align*}
h^0(Y,\w_Y(2N-3A))&=h^0(Y,\O_Y(2N-A))=h^1(Y,\O_Y(2N-A))\\
&=h^0(Y,\w_Y(A-2N))=h^0(Y,\w_Y(-2N))=0,
\end{align*}
where the last equality follows from the fact that $N$ is not a 
Weierstrass point of $Y$. Then there must be a section $s''$ of 
$\L''$, forming a $k[[t]]$-basis with $t^2s_1,t^2s_2$, whose 
restriction to $Y$ vanishes at $A$ with multiplicity 1. 
Up to replacing $s_3$ by $s_3-(s_3/t^2s_1)(A)t^2s_1$, 
we may assume that $s''=s_3$. Thus we may assume
$$
s_3\in t^2(u+(t,u^2))\psi.
$$

It follows that $D$ is given at $A$ by the maximal minors of a matrix 
whose entries belong to the corresponding entries of the matrix below:
$$
\left[\begin{matrix}
\{1\} & u^2+(t,u^3) & t^2(u+(t,u^2))\\
\{0\} & 2u+(t,u^2) & t^2+t^2(t,u)
\end{matrix}\right].
$$
Then the minors belong to
$$
2u+(t,u^2),\quad t^2+(t,u)^3,\quad (t,u)^3.
$$
Since the characteristic of $k$ is not 2, 
it follows that $D$ has multiplicity 2 at $A$.

Let us establish the multiplicity in \eqref{mult3} now. Let $t$ be a 
local parameter of 
$\O_{S,s}$, where $s:=\pi(N)$. 
Let $x$ (resp. $y$) 
be a local equation for $Y$ (resp. $X$) at $N$. We may choose them 
such that $t=xy$ in the local ring $\O_{C,N}$. Since the restriction maps
$$
H^0(C,\L)\to H^0(X,\L|_X)\quad\text{and}\quad H^0(C,\L)\to H^0(Y,\L|_Y)
$$
are surjective, $\L|_X\cong\O_X(2N)$ and $\L|_Y\cong\w_Y$, there is a 
basis $s_1,s_2,s_3$ of $H^0(C,\L)$ as a $k[[t]]$-module such that 
$s_1(N)\neq 0$, that $s_2|_Y=0$ and $s_2|_X$ vanishes at $N$, 
necessarily to order 2, and $s_3|_X=0$ and $s_3|_Y$ vanishes at $N$ to 
order 1. Thus, letting $\psi$ be the germ of $s_1$ at $N$, we may assume that
\begin{align*}
s_1\in&\{\psi\},\\
s_2\in&x(x+(y,x^2))\psi,\\
s_3\in&y(1+(x,y))\psi.
\end{align*}  

Now, $\w_{C/S}$ is generated at $N$ by the meromorphic differential 
$\tau:=dx/x=-dy/y$. 
This means that the canonical derivation $\partial$ on 
$\L_N$ induced by the composition of the universal derivation 
$\O_C\to\Omega^1_{C/S}$ with the canonical class 
$\Omega^1_{C/S}\to\w_{C/S}$ satisfies 
$\partial(x)=x\tau$ and $\partial(y)=-y\tau$. 

Thus, $D$ is given at $N$ by the maximal minors of a matrix 
whose entries belong to the corresponding entries of the matrix below:
$$
\left[\begin{matrix}
\{1\} & x(x+(y,x^2)) & y(1+(x,y))\\
\{0\} & 2x^2+x(y,x^2) & -y+y(x,y)
\end{matrix}\right].
$$
Then the minors belong to
$$
2x^2+kxy+(x,y)^3,\quad -y+(x,y)^2,\quad (x,y)^3.
$$
Since the characteristic of $k$ is not 2, 
it follows that $D$ has multiplicity 2 at $N$.

Finally, let us establish the multiplicity in \eqref{mult2}. Let $M$ 
denote the node of $Z$. Let $t$ be a local parameter of 
$\O_{S,s}$, where $s:=\pi(M)$. 
Since $Z$ is a node of the special fiber, 
there are local parameters $x$ and $y$ at $N$ such that 
$t\equiv xy\mod (x,y)^3$. Let $\wt Z$ be the normalization of $Z$ and 
$M_1,M_2\in\wt Z$ the points above $M$. The normalization map induces a 
canonical isomorphism
$$
H^0(Z,\w_Z)\lra H^0(\wt Z,\w_{\wt Z}(M_1+M_2)).
$$
Since the restriction map $H^0(C,\w_{C/S})\to H^0(Z,\w_Z)$ is surjective, 
there is a basis $s_1,s_2,s_3$ of $H^0(C,\L)$ as a $k[[t]]$-module such 
that $s_1(M)\neq 0$ and the pullbacks $\wt s_2$ and $\wt s_3$ of 
$s_2$ and $s_3$ to $\wt Z$ are such that both vanish at $M_1$ and $M_2$ 
but $\wt s_2$ vanishes at $M_1$ to order 2 and $\wt s_3$ vanishes at 
$M_2$ to order 2. Thus, letting $\psi$ be the germ of $s_1$ at $M$, we may 
assume
\begin{align*}
s_1\in&\{\psi\},\\
s_2\in&(x+(x,y)^2)\psi,\\
s_3\in&(y+(x,y)^2)\psi.
\end{align*}  

Let $\tau$ be a generator of $\w_{C/S}$ at $M$, and let $\partial$ be 
the canonical derivation of $\w_{C/S,M}$, 
induced by the composition of the universal derivation 
$\O_C\to\Omega^1_{C/S}$ with the canonical class 
$\Omega^1_{C/S}\to\w_{C/S}$. Choosing $\tau$ appropriately, we have 
$\partial(x)\equiv x \mod (x,y)^2$ and $\partial(y)\equiv -y \mod (x,y)^2$.

Thus, $D$ is given at $M$ by the maximal minors of a matrix 
whose entries belong to the corresponding entries of the matrix below:
$$
\left[\begin{matrix}
\{1\} & x+(x,y)^2 & y+(x,y)^2\\
\{0\} & x+(x,y)^2 & -y+(x,y)^2
\end{matrix}\right].
$$
Then the minors belong to
$$
x+(x,y)^2,\quad -y+(x,y)^2,\quad (x,y)^2,
$$
and thus $D$ has multiplicity 1 at $M$.
\end{proof}

\vspace{0.5cm}

{\smallsc Instituto de Matem\'atica Pura e Aplicada, 
Estrada Dona Castorina 110, 22460-320 Rio de Janeiro RJ, Brazil}

{\smallsl E-mail address: \small\verb?esteves@impa.br?}


\begin{thebibliography}{9999}

\bibitem{Diaz} S. Diaz,
\emph{Porteous's formula for maps between coherent sheaves.}
Michigan Math. J. {\bf 52} (2004), no. 3, 507--514.

\bibitem{E1} E. Esteves, 
\emph{Wronski algebra systems on families of singular curves.} 
Ann. Sci. \'Ecole Norm. Sup. (4) {\bf 29} (1996), no. 1, 107--134.

\bibitem{E2} E. Esteves,
\emph{Jets of singular foliations.}
Available at http://arxiv.org/abs/math/0611528.

\bibitem{F} W. Fulton,
\emph{Intersection Theory.}
Ergebnisse der Mathematik und ihrer Grenzgebiete. 3. Folge, vol. 2, 
2nd ed., Springer-Verlag, Berlin, 1998.

\bibitem{G1} L. Gatto,
\emph{$k$-forme wronskiane, successioni di pesi e punti di 
Weierstrass su curve di Gorenstein.}
Tesi di Dottorato, Universit\`a di Torino, 1993.

\bibitem{G2} L. Gatto,
\emph{Weight sequences versus gap sequences at singular points of 
Gorenstein curves.}
Geom. Dedicata {\bf 54} (1995), no. 3, 267--300.

\bibitem{HM} J. Harris and I. Morrison,
\emph{Moduli of curves.}
Graduate Texts in Mathematics, vol. 187, 
Springer-Verlag, New York, 1998.

\bibitem{LT1} D. Laksov and A. Thorup,
\emph{The algebra of jets.} 
Michigan Math. J. {\bf 48} (2000), 393--416.

\bibitem{LT2} D. Laksov and A. Thorup,
\emph{Wronski systems for families of local complete 
intersection curves.}
Comm. Algebra {\bf 31} (2003), no. 8, 4007--4035.

\bibitem{Mum} D. Mumford,
\emph{Towards an enumerative geometry of the moduli
space of curves}. In: Arithmetic and geometry, Vol.~II, 271--328, 
Progr. Math., 36, Birkh\"auser Boston, Boston, 1983.

\end{thebibliography}
\end{document}